\documentclass[11pt,a4paper]{article}

\usepackage{epsf,epsfig,amsfonts,amsgen,amsmath,amstext,amsbsy,amsopn,amsthm,cases,listings,color
}
\usepackage{ebezier,eepic}
\usepackage{color}
\usepackage{multirow}
\usepackage{epstopdf}
\usepackage{graphicx}
\usepackage{pgf,tikz}
\usepackage{mathrsfs}
\usepackage[marginal]{footmisc}
\usepackage{enumitem}
\usepackage{appendix}

\usepackage{pgf,tikz,pgfplots}
\usepackage{subfigure}
\usepackage{mathrsfs}
\usetikzlibrary{arrows}

\allowdisplaybreaks[1]

\definecolor{uuuuuu}{rgb}{0.27,0.27,0.27}
\definecolor{sqsqsq}{rgb}{0.1255,0.1255,0.1255}

\newtheorem{definition}{Definition} [section]
\newtheorem{theorem}[definition]{Theorem}
\newtheorem{lemma}[definition]{Lemma}
\newtheorem{proposition}[definition]{Proposition}

\newtheorem{conjecture}[definition]{Conjecture}
\newtheorem{claim}[definition]{Claim}

\newtheorem{observation}[definition]{Observation}

\newenvironment{pf}{\noindent{\bf Proof}}

\begin{document}
\title{\bf\Large A note on hypergraphs without non-trivial intersecting subgraphs}

\date{\today}

\author{Xizhi Liu\thanks{Department of Mathematics, Statistics, and Computer Science, University of Illinois, Chicago, IL, 60607 USA.
Email: xliu246@uic.edu.
Research partially supported by NSF award DMS-1763317.}
}
\maketitle
\begin{abstract}
A hypergraph $\mathcal{F}$ is non-trivial intersecting if every two edges in it have a nonempty intersection
but no vertex is contained in all edges of $\mathcal{F}$.
Mubayi and Verstra\"{e}te showed that for every $k \ge d+1 \ge 3$ and $n \ge (d+1)n/d$ every $k$-graph
$\mathcal{H}$ on $n$ vertices without a non-trivial intersecting subgraph of size $d+1$ contains at most $\binom{n-1}{k-1}$
edges. They conjectured that the same conclusion holds for all $d \ge k \ge 4$ and sufficiently large $n$.
We confirm their conjecture by proving a stronger statement.

They also conjectured that for $m \ge 4$ and sufficiently large $n$
the maximum size of a $3$-graph on $n$ vertices without a non-trivial intersecting subgraph of size $3m+1$
is achieved by certain Steiner systems.
We give a construction with more edges showing that their conjecture is not true in general.
\end{abstract}

\section{Introduction}
We use $[n]$ to denote the set $\{1,\ldots,n\}$
and use $\binom{V}{k}$ to denote the collection of all $k$-subsets of some set $V$.
For a hypergraph $\mathcal{H}$ we use $V(\mathcal{H})$ to denote the vertex set of $\mathcal{H}$
and use $|\mathcal{H}|$ to denote the number of edges in $\mathcal{H}$.

For $d \ge 2$ a hypergraph $\mathcal{F}$ is $d$-wise-intersecting if $\bigcap_{i\in [d]}E_i \neq \emptyset$
for all $E_1,\ldots,E_d \in \mathcal{F}$,
and $\mathcal{F}$ is non-trivial $d$-wise-intersecting if it is $d$-wise-intersecting but $\bigcap_{E\in \mathcal{F}}E = \emptyset$.
If $d=2$, then we simply call $\mathcal{F}$ intersecting and non-trivial intersecting, respectively.

A $d$-simplex is a collection of $d+1$ sets $\{A_1,\ldots,A_{d+1}\}$ such that
$\bigcap_{i\neq j}A_{i} \neq \emptyset$ for all $j \in [d+1]$, but $\bigcap_{i\in [d+1]}A_i = \emptyset$.
The Chv\'{a}tal Simplex Conjecture \cite{CH74} states that
for every $k \ge d+1 \ge 3$ and $n \ge (d+1)n/d$ if a hypegraph $\mathcal{H}\subset \binom{[n]}{k}$
does not contain a $d$-simplex as a subgraph, then $|\mathcal{H}| \le \binom{n-1}{k-1}$,
with equality only if $\mathcal{H}$ is a star, i.e. all sets in $\mathcal{H}$ contain a fixed vertex.
The case $k = d+1$ was proved by Chv\'{a}tal \cite{CH74}.
Mubayi and Verstra\"{e}te \cite{MV05} proved the conjecture for all $k \ge 3$ and $d=2$.
Recently, Currier \cite{CU20} proved this conjecture for all $k \ge d+1 \ge 3$ and $n \ge 2k$.
The Chv\'{a}tal Simplex Conjecture is still open in general for $n < 2k$ and $3 \le d \le k-2$,
and we refer the reader to \cite{BF77,CK99,FF87,FR76,FR81,KM10,KL17,LI20}
and their references for more results related to this conjecture.

It is easy to see that the family of all $d$-simplexes is the same as
the family of all non-trivial $d$-wise-intersecting hypergraphs of size $d+1$,
and if a hypergraph is $d$-wise-intersecting, then it is also $d'$-wise-intersecting for all $2\le d' \le d$.

In the proof for the Chv\'{a}tal Simplex Conjecture for $d=2$ Mubayi and Verstra\"{e}te
actually proved the following stronger result.

\begin{theorem}[Mubayi and Verstra\"{e}te \cite{MV05}]\label{THM-MV-05-A}
Let $k \ge d+1 \ge 3$ and $n \ge (d+1)n/d$.
Suppose that $\mathcal{H} \subset \binom{[n]}{k}$ contains no non-trivial intersecting subgraph of size $d+1$.
Then $|\mathcal{H}| \le \binom{n-1}{k-1}$, with equality only if $\mathcal{H}$ is a star.
\end{theorem}

Mubayi and Verstra\"{e}te also remarked that their proof of Theorem \ref{THM-MV-05-A}
actually works for $d$ slightly greater than $k$ as well, and they posed the following conjecture.

\begin{conjecture}[Mubayi and Verstra\"{e}te \cite{MV05}]\label{CONJ-MV-05}
Let $d \ge k \ge 4$ and $n$ be sufficiently large.
Suppose that $\mathcal{H} \subset \binom{[n]}{k}$ contains no non-trivial intersecting subgraph of size $d+1$.
Then $|\mathcal{H}| \le \binom{n-1}{k-1}$,
with equality only if $\mathcal{H}$ is a star.
\end{conjecture}

Let $m \ge 2$.
A Steiner $(n,3,m-1)$-system is a $3$-graph $\mathcal{S}$ on $n$ vertices such that
every pair of vertices in $V(\mathcal{S})$ is contained in exactly $m-1$ edges of $\mathcal{S}$.
It follows from Keevash's result \cite{KEE18B} that if $n$ is a multiple of $3$ and sufficiently large,
then there exists a Steiner $(n,3,m-1)$-system.

Notice that a Steiner $(n,3,m-1)$-system has size $\frac{m-1}{3}\binom{n}{2}$, which is greater than $\binom{n-1}{2}$ when $m \ge 4$.
It was observed by Mubayi and Verstra\"{e}te \cite{MV05} that
a Steiner $(n,3,m-1)$-system does not contain non-trivial intersecting subgraph of size $3m+1$.
Therefore, they made the following conjecture for $3$-graphs.

\begin{conjecture}[Mubayi and Verstra\"{e}te \cite{MV05}]\label{CONJ-MV-05-B}
Let $m \ge 4$ and $n$ be sufficiently large.
Suppose that $\mathcal{H} \subset \binom{[n]}{3}$ contains no non-trivial intersecting family of size $3m+1$.
Then $|\mathcal{H}| \le \frac{m-1}{3}\binom{n}{2}$, with equality holds iff $\mathcal{H}$ is a Steiner  $(n,3,m-1)$-system.
\end{conjecture}

In this note, we confirm Conjecture \ref{CONJ-MV-05} by proving a stronger statement (Theorem \ref{THM-xd-cluster-free}),
and disprove Conjecture \ref{CONJ-MV-05-B} by showing a construction with more than $\frac{m-1}{3}\binom{n}{2}$
edges and contains no non-trivial intersecting subgraph of size $3m+1$.

Let $s\ge 2$. A family $\mathcal{D} = \{D_1,\ldots,D_s\}$ is a $\Delta$-system (or a sunflower)
if $D_i \cap D_j = C$ for all $\{i,j\} \subset [s]$.
The set $C$ is called the center of $\mathcal{D}$.

\begin{definition}
Let $k, d\ge p \ge 2$, and $\vec{a} = (a_1,\ldots, a_p)$, $\vec{b} = (b_1,\ldots, b_p)$ be two sequences of positive integers
with $\sum_{i=1}^{p}a_i = k$.
\begin{itemize}
\item[(1)] An $\vec{a}$-partition of a $k$-set $E$ is a partition $E = \bigcup_{i\in [p]}A_i$
    such that $|A_i| = a_i$ for $i \in [p]$.
\item[(2)] A semi-$(\vec{a},\vec{b})$-$\Delta$-system is a collection of sets
    $\{E_0,E_{1}^{1},\ldots, E_{1}^{b_1},\ldots, E_{p}^{1},\ldots, E_{p}^{b_p}\}$
    such that for some $\vec{a}$-partition of $E_0=\bigcup_{i\in [p]}A_i$,
    the family $\{E_0, E_{i}^{1},\ldots, E_{i}^{b_i}\}$ is a $\Delta$-system
    with center $E_0\setminus A_i$ for all $i\in [p]$.
    The set $E_0$ is called the host of this semi-$(\vec{a},\vec{b})$-$\Delta$-system.
\item[(3)] An $(\vec{a},\vec{b})$-$\Delta$-system is a semi-$(\vec{a},\vec{b})$-$\Delta$-system
    $\{E_0,E_{1}^{1},\ldots, E_{1}^{b_1},\ldots, E_{p}^{1},\ldots, E_{p}^{b_p}\}$
    such that sets $E_{1}^{1}\setminus E_0,\ldots, E_{1}^{b_1}\setminus E_0,\ldots, E_{p}^{1}\setminus E_0,\ldots, E_{p}^{b_p}\setminus E_0$
    are pairwise disjoint.
\item[(4)] An $(\vec{a},d)$-$\Delta$-system is a $(\vec{a},\vec{b})$-$\Delta$-system for some $\vec{b}$ such that $\sum_{i=1}^{p}b_i = d$.
\end{itemize}
\end{definition}

From the definitions one can easily obtain the following observation.

\begin{observation}\label{OBS-p-greater-2-implies-nontrivial}
Let $k, d\ge p \ge 3$ and $\vec{a} = (a_1,\ldots, a_p)$ be a sequence of integers with $\sum_{i=1}^{p}a_i = k$.
Then an $(\vec{a},d)$-$\Delta$-system is a non-trivial $(p-1)$-wise-intersecting hypergraph with $d+1$ edges.
\end{observation}

An $(\vec{a},d)$-$\Delta$-system in which $d = p$, i.e. $b_1 = \cdots = b_p = 1$
was studied by F\"{u}redi and \"{O}zkahya in \cite{FO11}.
In this note we employ a machinery (a complicated version of the delta-system method)
developed by them and even earlier by Frankl and F\"{u}redi \cite{FF85},
to obtain the following tight bound for the size of a hypergraph without $(\vec{a},d)$-$\Delta$-systems for all $d\ge p \ge 2$.

\begin{theorem}\label{THM-xd-cluster-free}
Let $k > p \ge 2$, $d \ge p$, and $\vec{a} = (a_1,\ldots,a_p)$ be a sequence of positive integers with $\sum_{i=1}^{p}a_i = k$.
Suppose that $n \ge n_0(k,d)$ is sufficiently large and $\mathcal{H} \subset \binom{[n]}{k}$
does not contain a $(\vec{a},d)$-$\Delta$-system as a subgraph.
Then $|\mathcal{H}| \le \binom{n-1}{k-1}$,
with equality only if $\mathcal{H}$ is a star.
\end{theorem}

\textbf{Remark:} Our proof of Theorem \ref{THM-xd-cluster-free} uses the delta-system method and
Theorem \ref{THM-FU83} due to F\"{u}redi, so our lower bound for $n_0(k,d)$ is at least
exponential in $k$ and $d$.
It would be interesting to find the minimum value of $n_0(k,d)$ such that the statement
in Theorem \ref{THM-xd-cluster-free} holds for all $n \ge n_0(k,d)$.

The following result is an immediate consequence of
Theorem \ref{THM-xd-cluster-free} and Observation \ref{OBS-p-greater-2-implies-nontrivial}.

\begin{theorem}\label{THM-nontrivial-p-intersect}
Let $k > p \ge 3$, $d \ge p$.
Suppose that $n \ge n_0(k,d)$ is sufficiently large and $\mathcal{H} \subset \binom{[n]}{k}$
does not contain a non-trivial $(p-1)$-wise-intersecting subgraph of size $d+1$.
Then $|\mathcal{H}| \le \binom{n-1}{k-1}$,
with equality only if $\mathcal{H}$ is a star.
\end{theorem}

Note that Conjecture \ref{CONJ-MV-05} is a special case of Theorem \ref{THM-nontrivial-p-intersect}, i.e. $p=3$.

We are also able to prove the following stability version of Theorem \ref{THM-xd-cluster-free}.

\begin{theorem}\label{THM-stability}
Let $k > p \ge 2$, $d \ge p$, and $\vec{a} = (a_1,\ldots,a_p)$ be a sequence of positive integers with $\sum_{i=1}^{p}a_i = k$.
For every $\delta > 0$ there exists $\epsilon > 0$ and $n_0(k,d,\delta)$ such that the following holds for all $n\ge n_0(k,d,\delta)$.
Suppose that $\mathcal{H} \subset \binom{[n]}{k}$ does not contain a $(\vec{a},d)$-$\Delta$-system as a subgraph,
and $|\mathcal{H}| \ge (1-\epsilon)\binom{n-1}{k-1}$.
Then there exists a vertex $v \in [n]$ such that $v$ is contained in all but at most $\delta n^{k-1}$ edges in $\mathcal{H}$.
\end{theorem}

For $3$-graphs the following result shows that Conjecture \ref{CONJ-MV-05-B} is not true in general.

\begin{theorem}\label{THM-counterexample}
Let $m \ge 4$, $n$ be a multiple of $3$ and sufficiently large.
Then there exists a $3$-graph $\widehat{\mathcal{S}}$ on $n$ vertices with $\frac{m-1}{3}\binom{n}{2} + \frac{n}{3}$
edges and contains no non-trivial intersecting subgraph of size $3m+1$.
\end{theorem}

In Section 2 we present a construction that proves Theorem \ref{THM-counterexample}.
In Section 3 we present some preliminary lemmas for the proof of Theorems \ref{THM-xd-cluster-free} and \ref{THM-stability},
and in Section 4 we prove Theorems \ref{THM-xd-cluster-free} and \ref{THM-stability}.

\section{Constructions}
In this section we give a construction to show that Conjecture \ref{CONJ-MV-05-B} is not true in general.
We need the following structural theorem of intersecting $3$-graphs due to Kostochka and Mubayi \cite{KM17}.
Define
\begin{itemize}
\item $EKR(n) = \left\{ A \in \binom{[n]}{3}: 1 \in A \right\}$.
\item $H_{0}(n) = \left\{ A \in \binom{[n]}{3}: |A\cap [3]|\ge 2 \right\}$.
\item $H_{1}(n) = \left\{ A \in \binom{[n]}{3}: 1\in A \text{ and } |A\cap \{2,3,4\}|\ge 1 \right\} \cup \{234\}$.
\item $H_{2}(n) = \left\{ A \in \binom{[n]}{3}: 1\in A \text{ and } |A\cap \{2,3\}|\ge 1 \right\} \cup \{234,235,145\}$.
\item $H_{3}(n) = \left\{ A \in \binom{[n]}{3}: \{1,2\}\in A \right\} \cup \{134,135,145,234,235,245\}$.
\item $H_{4}(n) = \left\{ A \in \binom{[n]}{3}: \{1,2\}\in A \right\} \cup \{ 134, 156, 235, 236, 245, 246 \}$.
\item $H_{5}(n) = \left\{ A \in \binom{[n]}{3}: \{1,2\}\in A \right\} \cup \{ 134, 156, 136, 235, 236, 246 \}$.
\end{itemize}

\begin{theorem}[Kostochka and Mubayi \cite{KM17}]\label{THM-MK17-structure-3-gp}
Every intersecting $3$-graph with at least $11$ edges is contained in one of
$EKR(n),H_{0}(n),H_{1}(n),\ldots,H_{5}(n)$.
\end{theorem}

For a $3$-graph $\mathcal{H}$ and $\{u,v\} \subset V(\mathcal{H})$ let $\deg_{\mathcal{H}}(uv)$ denote the number of edges in $\mathcal{H}$
that contain both $u$ and $v$.
Let $\Delta_{2}(\mathcal{H}) = \max\{\deg_{\mathcal{H}}(uv): \{u,v\} \subset V(\mathcal{H})\}$.

\begin{observation}\label{OBS-heavy-shadow-nontrivial-intersec}
Let $\mathcal{H}$ be a $3$-graph with $e$ edges.
If $\mathcal{H} \subset H_{0}(n)$, then $\Delta_{2}(\mathcal{H}) \ge \lceil \frac{e}{3} \rceil$.
If $\mathcal{H} \subset H_{2}(n)$, then $\Delta_{2}(\mathcal{H}) \ge \lceil \frac{e-3}{2} \rceil$.
If $\mathcal{H}$ is contained in $H_{3}(n)$, $H_{4}(n)$, or $H_{5}(n)$, then $\Delta_{2}(\mathcal{H}) \ge e-6$.
\end{observation}

Now we define the construction.
Let $n$ be a multiple of $3$ and sufficiently large.
Let $\mathcal{S} \subset \binom{[n]}{3}$ be a Steiner $(n,3,m-1)$-system.
Then the complement of $\mathcal{S}$, which is $\bar{\mathcal{S}}:= \binom{[n]}{3}\setminus \mathcal{S}$,
satisfies that $d_{\bar{\mathcal{S}}}(uv) = n-m+1$ for all $\{u,v\}\subset V(\mathcal{S})$.
Therefore, by the R\"{o}dl-Ruci\'{n}ski-Szemer\'{e}di Theorem \cite{RRS09},
$\bar{\mathcal{S}}$ contains a matching $\mathcal{M}$ with $n/3$ edges.
Let $\widehat{\mathcal{S}} = \mathcal{S} \cup \mathcal{M}$.
Then it is easy to see that
\begin{align}
|\widehat{\mathcal{S}}| = |\mathcal{S}| + |\mathcal{M}| = \frac{m-1}{3}\binom{n}{2} + \frac{n}{3}. \notag
\end{align}

The following proposition proves Theorem \ref{THM-counterexample}.

\begin{proposition}\label{PROP-construct-S}
Let $m \ge 4$.
Then $\widehat{\mathcal{S}}$ does not contain a non-trivial intersecting subgraph of size $3m+1$.
\end{proposition}
\begin{proof}
Suppose not.
Let $\mathcal{F} \subset \widehat{\mathcal{S}}$ be a non-trivial intersecting subgraph with $3m+1 \ge 11$ edges.
By Theorem \ref{THM-MK17-structure-3-gp}, $\mathcal{F}$ is contained in one of $H_0(n),H_{1}(n),\ldots,H_{5}(n)$.
Notice that $\Delta_2(\mathcal{F}) \le \Delta_2(\widehat{\mathcal{S}}) = m$.
If $\mathcal{F}$ is contained in one of $H_0(n), H_{2}(n),\ldots,H_{5}(n)$,
then by Observation \ref{OBS-heavy-shadow-nontrivial-intersec},
$\Delta_2(\mathcal{F}) \ge \min\left\{\lceil \frac{3m+1}{3} \rceil, \lceil \frac{3}{2}m-1 \rceil, 3m-5\right\} > m$, a contradiction.
Therefore, $\mathcal{F} \subset H_1(n)$.
Then $\mathcal{F}$ contains four vertices $v_0,v_1,v_2,v_3$ such that
$\deg_{\widehat{\mathcal{S}}}(v_0v_1) + \deg_{\widehat{\mathcal{S}}}(v_0v_2) + \deg_{\widehat{\mathcal{S}}}(v_0v_3) \ge 3m$,
which implies
$\deg_{\widehat{\mathcal{S}}}(v_0v_1) = \deg_{\widehat{\mathcal{S}}}(v_0v_2) = \deg_{\widehat{\mathcal{S}}}(v_0v_3) = m$.
However, this is impossible because the set $\left\{ \{u,v\} \subset V(\mathcal{S}): \deg_{\widehat{S}}(uv)=m \right\}$
consists of $n/3$ copies of pairwise vertex-disjoint triangles.
\end{proof}

\section{Lemmas}
In this section we present some preliminary lemmas for the proofs of Theorems \ref{THM-xd-cluster-free} and \ref{THM-stability}.
Our first lemma shows that a sufficiently large semi-$(\vec{a},\vec{c})$-$\Delta$-system
contains an $(\vec{a},\vec{b})$-$\Delta$-system.

\begin{lemma}\label{LEMMA-semi-implies-nonsemi}
Let $k, d\ge p \ge 2$ and $\vec{a} = (a_1,\ldots, a_p)$, $\vec{b} = (b_1,\ldots, b_p)$, $\vec{c} = (c_1,\ldots, c_p)$
be sequences of positive integers with $\sum_{i=1}^{p}a_i = k$.
Suppose that $c_i \ge b_i + \sum_{j=1}^{i-1}a_jb_j$ for $i \in [p]$.
Then every semi-$(\vec{a},\vec{c})$-$\Delta$-system contains an $(\vec{a},\vec{b})$-$\Delta$-system.
In particular, if $c_1 \ge 1$ and $c_i \ge k d$ for $2 \le i \le p$,
then every semi-$(\vec{a},\vec{c})$-$\Delta$-system contains an $(\vec{a},d)$-$\Delta$-system.
\end{lemma}
\begin{proof}
Let $\mathcal{F} = \{E_0,E_{1}^{1},\ldots, E_{1}^{c_1},\ldots, E_{p}^{1},\ldots, E_{p}^{c_p}\}$ be a semi-$(\vec{a},\vec{c})$-$\Delta$-system.
Our goal is to choose $\{F_i^{1},\ldots,F_{i}^{b_i}\} \subset \{E_i^{1},\ldots,E_{i}^{c_i}\}$ for all $i \in [p]$
so that sets $E_0$, $F_{1}^{1},\ldots, F_{1}^{b_1}$, $\ldots$, $F_{p}^{1},\ldots, F_{p}^{b_p}$ form a $(\vec{a},\vec{b})$-$\Delta$-system.

Since $c_1 \ge b_1$, we can simply let $F_{1}^{j} = E_{1}^{j}$ for $j \in [b_1]$.
Now suppose that we have defined sets $\{F_{1}^{1},\ldots, F_{1}^{b_1},\ldots, F_{i}^{1},\ldots, F_{i}^{b_i}\}$
for some $i\in [p-1]$.
We are going to define sets $F_{i+1}^{1},\ldots, F_{i+1}^{b_{i+1}}$.
Note that for every $1 \le j \le i$ and $1 \le \ell \le b_j$ the set $F_{j}^{\ell}\setminus E_0$
can have nonempty intersection with at most $a_j$ sets in $\{E_{i+1}^{1},\ldots, E_{i+1}^{c_{i+1}}\}$.
Since $c_{i+1} \ge b_{i+1} + \sum_{j=1}^{i}a_jb_j$, there exist at least $b_{i+1}$ sets in $\{E_{i+1}^{1},\ldots, E_{i+1}^{c_{i+1}}\}$
that have empty intersection with all sets in
$\{F_{1}^{1}\setminus E_0,\ldots, F_{1}^{b_1}\setminus E_0,\ldots, F_{i}^{1}\setminus E_0,\ldots, F_{i}^{b_i}\setminus E_0\}$,
and choose any $b_i$ sets from them to form $\{F_{i+1}^{1},\ldots, F_{i+1}^{b_{i+1}}\}$.
The process terminates when $i=p$, and clearly, sets $E_0,F_{1}^{1},\ldots, F_{1}^{b_1},\ldots, F_{p}^{1},\ldots, F_{p}^{b_p}$ form an $(\vec{a},\vec{b})$-$\Delta$-system.

Now suppose that $c_1 \ge 1$ and $c_i \ge k d$ for $2 \le i \le p$.
Let $b_1 = 1$ and $b_i \ge 1$ for $2 \le i \le p$ such that $\sum_{i=2}^{p}b_i = d-1$.
Since $c_i \ge k d \ge b_i + \sum_{j=1}^{i-1}a_jb_j$, by the previous argument,
$\mathcal{F}$ contains an $(\vec{a},\vec{b})$-$\Delta$-system, which is an $(\vec{a},d)$-$\Delta$-system.
\end{proof}

For a hypergraph $\mathcal{H}$ and $E \in \mathcal{H}$.
The intersection structure of $E$ with respect to $\mathcal{H}$ is
\begin{align}
\mathcal{I}(E,\mathcal{H}) := \{E \cap E': E' \in \mathcal{H}\setminus\{E\}\}. \notag
\end{align}
A hypergraph $\mathcal{H} \subset \binom{[n]}{k}$ is $k$-partite if there exists a partition
$[n] = V_1\cup \cdots \cup V_k$ such that $|E\cap V_i| = 1$ for all $i\in [k]$.
Suppose that $\mathcal{H}$ is $k$-partite with $k$ parts $V_1,\ldots,V_k$.
Then for every $S \subset [n]$, its projection is $\Pi(S):= \{i: S \cap V_i \neq \emptyset\}$.
For every family $\mathcal{F} \subset 2^{[n]}$, its projection is $\Pi(\mathcal{F}):= \{\Pi(F): F\in \mathcal{F}\}$.

\begin{definition}
Let $s \ge 2$.
A hypergraph $\mathcal{H} \subset \binom{[n]}{k}$ is $s$-homogeneous if it satisfies the following conditions.
\begin{itemize}
\item[(1)] $\mathcal{H}$ is $k$-partite.
\item[(2)] There exists a family $\mathcal{J} \subset 2^{[k]}\setminus\{[k]\}$ such that
           $\Pi(\mathcal{I}(E,\mathcal{H})) = \mathcal{J}$ for all $E \in \mathcal{H}$,
           where $\mathcal{J}$ is called the intersection pattern of $\mathcal{H}$.
\item[(3)] $\mathcal{J}$ is closed under intersection, i.e. if $A,B\in \mathcal{J}$, then $A \cap B \in \mathcal{J}$.
\item[(4)] For every $E \in \mathcal{H}$ every set in $\mathcal{I}(E,\mathcal{H})$ is the center of a $\Delta$-system $\mathcal{D}$
           of size $s$ formed by edges of $\mathcal{H}$ and containing $E$, i.e. $E\in \mathcal{D} \subset \mathcal{H}$.
\end{itemize}
A hypergraph $\mathcal{H} \subset \binom{[n]}{k}$ is homogeneous if it is $s$-homogeneous for some $s\ge 2$.
\end{definition}

F\"{u}redi \cite{FU83} showed that for every $s \ge 2$, every hypergraph contains a large $s$-homogeneous subgraph.

\begin{theorem}[F\"{u}redi \cite{FU83}]\label{THM-FU83}
For every $s,k \ge 2$, there exists a constant $c(k,s) > 0$ such that every
hypergraph $\mathcal{H} \subset \binom{[n]}{k}$ contains a $s$-homogeneous subgraph $\mathcal{H}^{\ast}$
with $|\mathcal{H}^{\ast}| \ge c(k,s)|\mathcal{H}|$.
\end{theorem}

For a family $\mathcal{J} \subset 2^{[k]}\setminus\{[k]\}$ the rank of $\mathcal{J}$ is
\begin{align}
r(\mathcal{J}) := \min\{|A|: A\subset [k], A\not\in \mathcal{J} \text{ and  $\not\exists B \in \mathcal{J}$ such that $A \subset B$} \}. \notag
\end{align}
It is easy to see from the definition that $r(\mathcal{J}) = k$ iff $\mathcal{J} = 2^{[k]}\setminus\{[k]\}$.

For a hypergraph $\mathcal{H}\subset \binom{[n]}{k}$ and $1 \le i \le k-1$ the $i$-th shadow of $\mathcal{H}$ is
\begin{align}
\partial_{i}\mathcal{H}:= \left\{A\in\binom{[n]}{k-i}: \exists E\in \mathcal{H} \text{ such that } A\subset E\right\}. \notag
\end{align}
For convention, let $\partial_{0}\mathcal{H} = \mathcal{H}$.

The following lemma gives an upper bound for the size of a homogeneous hypergraph $\mathcal{H}$
in terms of the rank of its intersection pattern and its shadow.

\begin{lemma}\label{LEMMA-rank-up-bound}
Let $\mathcal{H} \subset \binom{[n]}{k}$ be a homogeneous hypergraph
with intersection pattern $\mathcal{J} \subset 2^{[k]}\setminus\{[k]\}$.
Then $|\mathcal{H}| \le |\partial_{k-r(\mathcal{J})}\mathcal{H}|$.
\end{lemma}
\begin{proof}
Let $r = r(\mathcal{J})$.
By the definition of rank, there exists an $r$-set $S\subset [k]$ that is not contained in $\mathcal{J}$,
and moreover, every $T\subset [k]$ that contains $S$ is also not contained in $\mathcal{J}$.
Since $\Pi(\mathcal{I}(E,\mathcal{H})) = \mathcal{J}$ for all $E \in \mathcal{H}$,
there exists an $r$-set in every $E \in \mathcal{H}$ that is not contained in any other edges in $\mathcal{H}$.
Therefore, $|\mathcal{H}| \le  |\partial_{k-r}\mathcal{H}|$.
\end{proof}

The following lemma shows that if a hypergraph is $s$-homogeneous for sufficiently large $s$
and does not contain an $(\vec{a},d)$-$\Delta$-system as a subgraph,
then the rank of its intersection pattern is at most $k-1$.

\begin{lemma}\label{LEMMA-rank-J-less-k}
Let $d \ge p \ge 2$, $k>p$, $s\ge kd+1$, and $\vec{a} = (a_1,\ldots, a_p)$
be a sequence of positive integers with $\sum_{i=1}^{p}a_i = k$.
Let $\mathcal{H} \subset \binom{[n]}{k}$ be a $s$-homogeneous hypergraph
with intersection pattern $\mathcal{J} \subset 2^{[k]}\setminus\{[k]\}$.
If $r(\mathcal{J}) = k$, then $\mathcal{H}$ contains an $(\vec{a},d)$-$\Delta$-system.
\end{lemma}
\begin{proof}
Since $r(\mathcal{J}) = k$, $\mathcal{J} = 2^{[k]}\setminus \{[k]\}$.
Let $E\in \mathcal{H}$ and let $\bigcup_{i=1}^{p}A_i = E$ be an $\vec{a}$-partition of $E$.
Since $\Pi(\mathcal{I}(E,\mathcal{H})) = \mathcal{J}$, we have $E\setminus A_i \in \mathcal{I}(E,\mathcal{H})$ for all $i \in [p]$.
Since $\mathcal{H}$ is $s$-homogeneous, there exists a $\Delta$-system $\mathcal{D}_i$ of size $s$ with center $E\setminus A_i$ for $i \in [p]$.
By assumption, $s \ge kd+1$, therefore, by Lemma \ref{LEMMA-semi-implies-nonsemi}, $\mathcal{H}$ contains an $(\vec{a},d)$-$\Delta$-system.
\end{proof}

Lemmas \ref{LEMMA-rank-J-less-k} and \ref{LEMMA-rank-up-bound}, and Theorem \ref{THM-FU83} implies that following proposition.

\begin{proposition}\label{PROP-shadow-H-H-F-free}
Let $d \ge p \ge 2$, $k > p$,
and $\vec{a} = (a_1,\ldots, a_p)$
be a sequence of positive integers with $\sum_{i=1}^{p}a_i = k$.
Let $\mathcal{H} \subset \binom{[n]}{k}$ be a hypergraph that contains no $(\vec{a},d)$-$\Delta$-systems.
Then there exists a constant $c(k,d) > 0$ such that
$|\partial\mathcal{H}| \ge c(k,d)|\mathcal{H}|$.
\end{proposition}
\begin{proof}
Let $s = kd+1$ and $\mathcal{H}^{\ast}$ be a maximum $s$-homogeneous subgraph of $\mathcal{H}$ with intersection pattern $\mathcal{J}$.
Then by Theorem \ref{THM-FU83}, $|\mathcal{H}^{\ast}| \ge c(k,d) |\mathcal{H}|$ for some constant $c(k,d) > 0$.
Since $\mathcal{H}^{\ast}$ contains no $(\vec{a},d)$-$\Delta$-systems, by Lemma \ref{LEMMA-rank-J-less-k}, $r(\mathcal{J}) \le k-1$.
So by Lemma \ref{LEMMA-rank-up-bound}, $|\mathcal{H}^{\ast}| \le |\partial\mathcal{H}^{\ast}|$.
Therefore, $|\partial\mathcal{H}| \ge |\partial\mathcal{H}^{\ast}| \ge |\mathcal{H}^{\ast}| \ge c(k,d)|\mathcal{H}|$.
\end{proof}

The next lemma gives another condition that guarantees a hypergraph to contain an $(\vec{a},d)$-$\Delta$-system as a subgraph.

\begin{lemma}\label{LEMMA-H-covered-Ai}
Let $d \ge p \ge 2$, $k > p$, $s \ge kd+1$, and $\vec{a} = (a_1,\ldots, a_p)$
be a sequence of positive integers with $\sum_{i=1}^{p}a_i = k$.
Let $\mathcal{H} \subset \binom{[n]}{k}$ and $\mathcal{H}^{\ast}$ be a $s$-homogeneous subgraph of $\mathcal{H}$.
Let $E_0 \in \mathcal{H}^{\ast}$ and $\bigcup_{i\in [p]}A_i = E_0$ be an $\vec{a}$-partition of $E_0$.
Suppose that there exists $i_0 \in [p]$ such that $E\setminus A_i \in \mathcal{I}(E,\mathcal{H}^{\ast})$ for all $i \in [p]\setminus\{i_0\}$,
and there exists $E\in \mathcal{H}$ such that $E \cap E_0 = E_0\setminus A_{i_0}$,
Then $\mathcal{H}$ contains an $(\vec{a},d)$-$\Delta$-system.
\end{lemma}
\begin{proof}
Without loss of generality, we may assume that $i_0=1$.
By assumption, $E\setminus A_i$ is the center of $\Delta$-system of size $s\ge kd +1$
in $\mathcal{H}^{\ast}\subset \mathcal{H}$ for $2 \le i \le k$,
and $E_0\setminus A_1$ is the center of a $\Delta$-system of size $2$ in $\mathcal{H}$, i.e. $\{E_0, E\}$.
Therefore, by Lemma \ref{LEMMA-semi-implies-nonsemi}, $\mathcal{H}$ contains an $(\vec{a},d)$-$\Delta$-system.
\end{proof}

\begin{lemma}\label{LEMMA-large-intersection}
Let $d\ge p \ge 2$, $k > p$, $s\ge kd+1$, and $\vec{a} = (a_1,\ldots, a_p)$
be a sequence of positive integers with $\sum_{i=1}^{p}a_i = k$.
Let $\mathcal{H} \subset \binom{[n]}{k}$ be a hypergraph
that does not contain $(\vec{a},d)$-$\Delta$-systems.
Let $\mathcal{H}^{\ast}$ be a $s$-homogeneous subgraph of $\mathcal{H}$ with intersection pattern $\mathcal{J}$.
Suppose that $r(\mathcal{J}) = k-1$ and $\mathcal{J}$ contains exactly $k-1$ $(k-1)$-sets.
Let $v\in E \in \mathcal{H}^{\ast}$ be the vertex that is contained in all $(k-1)$-sets in
$\mathcal{I}(E,\mathcal{H}^{\ast})$.
Then $v \in F$ for all $F \in \mathcal{H}$ that satisfies $|F\cap E|\ge k-a_1$
\end{lemma}
\begin{proof}
Let $E = \{v_1,\ldots,v_k\} \in \mathcal{H}^{\ast}$ and suppose that $v_1$ is contained in all $(k-1)$-sets in
$\mathcal{I}(E,\mathcal{H}^{\ast})$.
Let $F \in \mathcal{H}$ and suppose that $|E \cap F| = k-t$ for some $1 \le t \le a_1$, but $v_1\not\in F$.
If $t = a_1$, then let $\bigcup_{i\in[p]}A_i = E$ be an $\vec{a}$-partition such that
$A_1 = E\setminus F$.
For $2\le i \le p$ since $v_1\in E\setminus A_i$, $E\setminus A_i \in \mathcal{I}(E,\mathcal{H}^{\ast})$.
Therefore, $E\setminus A_i$ is the canter of a $\Delta$-system of size $s$ in $\mathcal{H}^{\ast}$ for $2 \le i \le p$.
Since $E\setminus A_1$ is the center of a $\Delta$-system of size $2$, i.e. $\{E,F\}$,
by Lemma \ref{LEMMA-semi-implies-nonsemi}, $\mathcal{H}$ contains an $(\vec{a},d)$-$\Delta$-system, a contradiction.
So, $t < a_1$.

Let $M\subset E$ such that $E\setminus F\subset M$ and $|M| = k-a_1+t$.
Since $|M|\le k-1$ and $v_1 \in M$, $M\in \mathcal{I}(E,\mathcal{H}^{\ast})$.
Therefore, $M$ is the center of a $\Delta$-system of size $s$ in $\mathcal{H}^{\ast}$,
which means that there exists $E_1\in \mathcal{H}^{\ast}$ such that $E_{1}\cap E = M$ and $(E_1\setminus E)\cap F = \emptyset$.
This implies that $E_1\cap F= M\setminus(E\setminus F)$ and $|E_1\cap F| = k-a_1$.
Since $\Pi(\mathcal{I}(E_1,\mathcal{H}^{\ast})) = \Pi(\mathcal{I}(E,\mathcal{H}^{\ast}))$,
applying the same argument as above to $E_1$ and $F$ we obtain that
$\mathcal{H}$ contains an $(\vec{a},d)$-$\Delta$-system, a contradiction.
\end{proof}

For a hypergraph $\mathcal{H}$ and $E\in \mathcal{H}$ the weight of $E$ is
\begin{align}
\omega_{\mathcal{H}}(E) := \sum_{E' \subset E, |E'|=k-1} \frac{1}{\deg_{\mathcal{H}}(E')}, \notag
\end{align}
where $\deg_{\mathcal{H}}(E')$ is the number of edges in $\mathcal{H}$ containing $E'$.
We have the following identity:
\begin{align}\label{EQU-weight-sum}
\sum_{E\in \mathcal{H}} \omega_{\mathcal{H}}(E)
= \sum_{E\in \mathcal{H}}\sum_{E' \subset E, |E'|=k-1} \frac{1}{\deg_{\mathcal{H}}(E')}
= \sum_{E' \in \partial\mathcal{H}}\sum_{E\in \mathcal{H}, E'\subset E}\frac{1}{\deg_{\mathcal{H}}(E')}
= |\partial\mathcal{H}|.
\end{align}

The following lemma gives a lower bound for $\omega_{\mathcal{H}}(E)$ regarding the intersection structure of $E$
in a homogeneous subgraph of $\mathcal{H}$.

\begin{lemma}\label{LEMMA-rank-k-1}
Let $d\ge p \ge 2$, $k > p$, $s\ge kd+1$, and $\vec{a} = (a_1,\ldots, a_p)$
be a sequence of positive integers with $\sum_{i=1}^{p}a_i = k$.
Let $\mathcal{H} \subset \binom{[n]}{k}$ be a hypergraph
that does not contain $(\vec{a},d)$-$\Delta$-systems.
Let $\mathcal{H}^{\ast}$ be a $s$-homogeneous subgraph of $\mathcal{H}$ with intersection pattern $\mathcal{J}$.
Suppose that $r(\mathcal{J}) = k-1$. Then the followings hold.
\begin{itemize}
\item[(1)] If $\mathcal{J}$ contains exactly $k-1$ $(k-1)$-sets,
            then every $E\in \mathcal{H}^{\ast}$ contains a $(k-1)$-subset that is not contained in any other edges in $\mathcal{H}$.
            In particular, $\omega_{\mathcal{H}}(E) \ge 1$ for all $E\in \mathcal{H}^{\ast}$.
\item[(2)] If $\mathcal{J}$ contains at most $k-2$ $(k-1)$-sets,
            then $\omega_{\mathcal{H}}(E) \ge \frac{k}{k-1}$ for all $E\in \mathcal{H}^{\ast}$.
\end{itemize}
\end{lemma}
\begin{proof}
We prove $(1)$ first.
We may assume that $a_1 \ge \cdots \ge a_k$, and note that $a_1 \ge 2$ since $\sum_{i=1}^{p}a_i = k>p$.
Let $E = \{v_1,\ldots,v_{k}\} \in \mathcal{H}^{\ast}$.
Since $\Pi(\mathcal{I}(E,\mathcal{H}^{\ast})) = \mathcal{J}$, by assumption,
there are exactly $k-1$ $(k-1)$-sets in $\mathcal{I}(E,\mathcal{H}^{\ast})$.
Without loss of generality we may assume that $E\setminus\{v_i\} \in \mathcal{I}(E,\mathcal{H}^{\ast})$ for $2 \le i \le k$.
We claim that $\{v_2,\ldots,v_k\}$ is not contain in any set in $\mathcal{H}\setminus \{E\}$.
Indeed, if there exists $E_1 \in \mathcal{H}$ such that $\{v_2,\ldots,v_k\} \subset E_1$,
then $|E_1\cap E| \ge k-1$.
So, by Lemma \ref{LEMMA-large-intersection}, $v_1\in E_1$, which implies that $E_1=E$.

Now we prove $(2)$.
Suppose that $\mathcal{J}$ has exactly $k-t$ $(k-1)$-sets for some $2\le t \le k$.
Let $E = \{v_1,\ldots,v_{k}\} \in \mathcal{H}^{\ast}$.
Without loss of generality, we may assume that $E\setminus\{v_i\} \in \mathcal{I}(E,\mathcal{H}^{\ast})$ for $t+1 \le i \le k$.

\begin{claim}\label{CLAIM-no-t-1-edges}
There does not exist a $(t-1)$-set $I\subset [t]$ and $t-1$ distinct vertices $\{u_i: i\in I\}$,
such that $(E\setminus\{v_i\})\cup \{u_i\} \in \mathcal{H}$ for all $i\in I$.
\end{claim}
\begin{proof}[Proof of Claim \ref{CLAIM-no-t-1-edges}]
Suppose not, and without loss of generality we may assume that
$F_i:= (E\setminus\{v_i\})\cup \{u_i\} \in \mathcal{H}$ for all $2 \le i \le t$,
where $u_2,\ldots,u_t$ are distinct vertices.

By assumption $\mathcal{I}(E,\mathcal{H}^{\ast})$ contains all $(k-1)$-sets that contain $\{v_1,\ldots,v_t\}$.
Since $\mathcal{I}(E,\mathcal{H}^{\ast})$ is closed under intersection, $\mathcal{I}(E,\mathcal{H}^{\ast})$ contains
all proper subsets of $E$ that contain $\{v_1,\ldots,v_t\}$,
i.e. if $A\subset \{v_{t+1},\ldots,v_k\}$, then $E\setminus A \in \mathcal{I}(E,\mathcal{H}^{\ast})$.

On the other hand, since $r(\mathcal{I}(E,\mathcal{H}^{\ast})) = k-1 \ge k-2$
and $E\setminus\{v_i\}, E\setminus\{v_j\} \not\in \mathcal{I}(E,\mathcal{H}^{\ast})$ for $i,j \in [t]$,
we have $E\setminus\{v_i,v_j\} \in \mathcal{I}(E,\mathcal{H}^{\ast})$ for all $\{i,j\}\subset [t]$.
This together with the previous argument and the property that $\mathcal{I}(E,\mathcal{H}^{\ast})$ is closed under intersection
imply that if $|A \cap \{v_1,\ldots,v_t\}| \ge 2$, then $E\setminus A \in \mathcal{I}(E,\mathcal{H}^{\ast})$.

Let $i_0 \in [p]$ such that $\sum_{i=1}^{i_0-1}a_i < t \le \sum_{i=1}^{i_0}a_i$,
and let $\ell = t - \sum_{i=1}^{i_0-1}a_i$.
Recall that $a_1 \ge \cdots \ge a_p\ge 1$ and $a_1 \ge 2$.
Suppose that $\ell \ge 2$.
Then there exists an $\vec{a}$-partition $E = \bigcup_{i\in[p]}A_i$
such that $A_1, \ldots, A_{i_0-1} \subset \{v_1, \ldots, v_t\}$, $|A_{i_0} \cap \{v_1, \ldots, v_t\}|\ge \ell \ge 2$,
and $A_{i_0+1},\ldots, A_p \subset \{v_{t+1},\ldots, v_k\}$.
Since $a_1 \ge \cdots \ge a_{i_0-1} \ge a_{i_0} \ge 2$, by the argument above, $E\setminus A_{i} \in \mathcal{I}(E,\mathcal{H}^{\ast})$
for all $i\in [p]$.
Therefore, $E\setminus A_{i}$ is the center of a $\Delta$-system of size $s \ge kd +1$ in $\mathcal{H}^{\ast}$ for $i\in [p]$,
so by Lemma \ref{LEMMA-semi-implies-nonsemi}, $\mathcal{H}$ contains an $(\vec{a},d)$-$\Delta$-system, a contradiction.
Therefore, $\ell = 1$.

Suppose that $a_{i_0} = 1$.
Then let $E = \bigcup_{i\in[p]}A_i$ be an $\vec{a}$-partition
such that $\bigcup_{i \in [i_0]}A_i = \{v_1,\ldots, v_t\}$ and $v_1 \in A_1$.
Since $A_1 \subset \{v_1,\ldots, v_t\}$ and $|A_1| \ge 2$, $E\setminus A_1 \in \mathcal{I}(E,\mathcal{H}^{\ast})$.
So $E\setminus A_1$ is the center of a $\Delta$-system of size $s$.
Without loss of generality we may assume that $a_2 = \cdots = a_{i_0} = 1$ since other cases can be proved similarly.
For $i_0+1 \le i \le p$ since $A_{i}\subset \{v_{t+1},\ldots,v_k\}$,
$E\setminus A_{i} \in \mathcal{I}(E,\mathcal{H}^{\ast})$.
So $E\setminus A_{i}$ is is the center of a $\Delta$-system of size $s$ for $i_0+1 \le i \le p$.
Notice that by assumption for every $2 \le i \le i_0$ there exists $F_{j_i} \in \mathcal{H}$
such that $F_{j_i} \cap E = E\setminus\{v_{j_i}\}$ for $2 \le j_i \le t$,
and moreover, $\{F_{j_i} \setminus E: 2 \le i \le i_0\}$ are distinct.
Therefore, by a similar argument as in the proof of Lemma \ref{LEMMA-semi-implies-nonsemi},
$\mathcal{H}$ contains an $(\vec{a},d)$-$\Delta$-system, a contradiction.
Therefore, $a_{i_0} \ge 2$.

Suppose that $a_1 \ge 3$.
Then let $E = \bigcup_{i\in[p]}A_i$ be an $\vec{a}$-partition
such that $A_2 \cup \cdots \cup A_{i_0-1} \subset \{v_1,\ldots, v_t\}$,
$v_{t+1}\in A_1$ and $A_1\setminus\{v_{t+1}\}\subset \{v_1,\ldots, v_t\}$,
and $\{v_1,\ldots, v_t\} \setminus \left(\bigcup_{i\in [i_0-1]}A_i\right) \subset A_{i_0}$.
Then $|A_i \cap \{v_1,\ldots, v_t\}| \ge 2$ for all $i \in [i_0]$
and $A_{j} \subset \{v_{t+1},\ldots,v_k\}$ for all $i_0+1 \le j \le p$.
Therefore, $E\setminus A_{i}$ is the center of a $\Delta$-system of size $s \ge kd +1$ in $\mathcal{H}^{\ast}$ for $i\in [p]$,
so by Lemma \ref{LEMMA-semi-implies-nonsemi}, $\mathcal{H}$ contains an $(\vec{a},d)$-$\Delta$-system, a contradiction.
Therefore, $a_1 = 2$.

Suppose that $a_{p} = 1$.
Then let $E = \bigcup_{i\in[p]}A_i$ be an $\vec{a}$-partition
such that $A_1 \cup \cdots \cup A_{i_0-1} = \{v_1,\ldots, v_{t-1}\}$ and $A_{p} = \{v_t\}$.
Then $E\setminus A_{i}$ is the center of a $\Delta$-system of size $s \ge kd +1$ in $\mathcal{H}^{\ast}$ for $i\in [p-1]$
and $E\setminus A_p$ is the center of a $\Delta$-system of size $2$ in $\mathcal{H}$, i.e. $\{E,F_{t}\}$.
Therefore, by Lemma \ref{LEMMA-semi-implies-nonsemi}, $\mathcal{H}$ contains an $(\vec{a},d)$-$\Delta$-system, a contradiction.
Therefore, $a_p = 2$.

Now we have $a_1 = \cdots = a_p =2$ and $\ell = 1$.
Then $t$ is odd, $t\ge 3$, and $k$ is even, $k> t$.
Since $E\setminus\{v_k\} \in \mathcal{I}(E,\mathcal{H}^{\ast})$,
$E\setminus\{v_k\}$ is the center of a $\Delta$-system of size $s$ in $\mathcal{H}^{\ast}$.
So there exists $F_k:= (E\setminus\{v_k\})\cup\{u_{k}\} \in \mathcal{H}^{\ast}$
such that $u_{k} \not\in \{u_2,\ldots,u_t\}$.
Let $F_k = \bigcup_{i\in[p]}A_i$ be an $\vec{a}$-partition
such that $A_1 = \{v_2,u_k\}$, $A_2 = \{v_1,v_3\}$, and $A_{i} = \{v_{2i-2},v_{2i-1}\}$ for $3\le i \le p$.
Then for every $i \in [p]\setminus\{1\}$, either $A_i \subset \{v_1,\ldots,v_t\}$ or $A_i \subset \{v_{t+1},\ldots,v_{k-1}\}$.
Since $\Pi(\mathcal{I}(F_k,\mathcal{H}^{\ast})) = \Pi(\mathcal{I}(E,\mathcal{H}^{\ast}))$,
$F_k\setminus A_{i}$ is the center of a $\Delta$-system of size $s \ge kd +1$ in $\mathcal{H}^{\ast}$ for $i \in [p]\setminus\{1\}$.
Since $V\setminus A_1$ is the center of a $\Delta$-system of size $2$ in $\mathcal{H}$, i.e. $\{F_k,F_2\}$.
Therefore, by Lemma \ref{LEMMA-semi-implies-nonsemi}, $\mathcal{H}$ contains an $(\vec{a},d)$-$\Delta$-system, a contradiction.
This completes the proof of Claim \ref{CLAIM-no-t-1-edges}.
\end{proof}

Define a bipartite graph $G$ with two parts $L = \{v_1,\ldots,v_t\}$ and $R = [n]\setminus E$,
and for every $v_i \in L$ and $u \in R$, $v_i u$ is an edge in $G$ iff $(E \setminus v_i) \cup \{u\} \in \mathcal{H}$.
Claim \ref{CLAIM-no-t-1-edges} implies that there are at most $t-2$ pairwise disjoint edges in $G$.
Therefore, by the K\"{o}nig-Hall theorem, $G$ contains a vertex cover $S$ with $|S| \le t-2$.
Let $\ell = |L\setminus S| \ge 2$.
Then $|S\cap R| \le \ell-2$.
For every $v \in L\setminus S$
since $N_{G}(v) \subset S\cap R$,
we obtain
\begin{align}
\deg_{\mathcal{H}}(E\setminus\{v\}) = \deg_{G}(v) + 1 \le \ell-1, \notag
\end{align}
which implies that
\begin{align}
\omega_{\mathcal{H}}(E) =
\sum_{E'\subset E, |E'|=k-1}\frac{1}{\deg_{\mathcal{H}}(E')} >
\sum_{v \in L\setminus S} \frac{1}{\deg_{\mathcal{H}}(E\setminus \{v\})} \ge \frac{\ell}{\ell-1} \ge \frac{k}{k-1}. \notag
\end{align}
This completes the proof of Lemma \ref{LEMMA-rank-k-1}.
\end{proof}

\section{Proofs}
In this section we prove Theorems \ref{THM-xd-cluster-free} and \ref{THM-stability}.
First, let us prove Theorem \ref{THM-stability}.

\begin{proof}[Proof of Theorem \ref{THM-stability}]
Let $k > p \ge 2$, $d \ge p$, and $\vec{a} = (a_1, \ldots, a_p)$ be a sequence of integers such that
$a_1 \ge \cdots \ge a_p \ge 1$ and $\sum_{i\in [p]}a_i = k$.
Let $\epsilon > 0$ and $n$ be sufficiently large.
Let $\mathcal{H}\subset \binom{[n]}{k}$ be a hypergraph that contains no $(\vec{a},d)$-$\Delta$-systems
and $|\mathcal{H}| \ge (1-\epsilon)\binom{n-1}{k-1}$.

Let $s = kd+1$ and let $\mathcal{H}_{1}$ be a maximum $s$-homogeneous subgraph of $\mathcal{H}$.
Suppose now we have defined $\mathcal{H}_1,\ldots,\mathcal{H}_{i}$ for some $i \ge 1$.
Let $\mathcal{H}_{i+1}$ be the maximum $s$-homogeneous subgraph of $\mathcal{H}\setminus\left(\bigcup_{j=1}^{i}\mathcal{H}_j\right)$.
This process terminates if $\mathcal{H}\setminus\left(\bigcup_{j=1}^{m}\mathcal{H}_j\right) = \emptyset$
or the intersection pattern of $\mathcal{H}_{m+1}$ has rank at most $k-2$ for some $m\ge 1$.
Let $\mathcal{J}_{i}$ denote the intersection pattern of $\mathcal{H}_{i}$ for $i\in [m]$,
and note that by definition and Lemma \ref{LEMMA-rank-J-less-k}, $r(\mathcal{J}_i) = k-1$ for $i\in [m]$.
Let
\begin{align}
\widehat{\mathcal{H}}_{1} & = \bigcup_{i}\left\{\mathcal{H}_i:
                            i\in [m] \text{ and $\mathcal{J}_i$ contains exactly $k-1$ $(k-1)$-sets} \right\}, \notag\\
\widehat{\mathcal{H}}_{2} & = \bigcup_{i}\left\{\mathcal{H}_i:
                            i\in [m] \text{ and $\mathcal{J}_i$ contains at most $k-2$ $(k-1)$-sets} \right\}, \notag\\
\widehat{\mathcal{H}}_{3}
& = \mathcal{H}\setminus\left(\widehat{\mathcal{H}}_{1}\cup \widehat{\mathcal{H}}_{2}\right)
= \mathcal{H}\setminus\left(\bigcup_{i\in [m]}\mathcal{H}_i\right). \notag
\end{align}
Our first step is to show that the sizes of $\widehat{\mathcal{H}}_{2}$ and $\widehat{\mathcal{H}}_{3}$ are small.

\begin{claim}\label{CLAIM-H2+H3}
$|\widehat{\mathcal{H}}_{2}| + |\widehat{\mathcal{H}}_{3}| < 3 \epsilon k \binom{n-1}{k-1}$.
\end{claim}
\begin{proof}[Proof of Claim \ref{CLAIM-H2+H3}]
First we show that $\widehat{\mathcal{H}}_{3} = O(n^{k-2})$.
We may assume that $\widehat{\mathcal{H}}_{3} \neq \emptyset$.
Recall that $\mathcal{H}_{m+1}$ is a maximum $s$-homogeneous subgraph of $\widehat{\mathcal{H}}_{3}$
with intersection pattern $\mathcal{J}_{m+1}$.
By Theorem \ref{THM-FU83}, there exists a constant $c(k,s)>0$ such that $|\mathcal{H}_{m+1}| \ge c(k,s)|\widehat{\mathcal{H}}_{3}|$.
By definition, $r(\mathcal{J}_{m+1}) \le k-2$, so by Lemma \ref{LEMMA-rank-up-bound},
$|\mathcal{H}_3| \le |\partial_2\mathcal{H}_3| \le \binom{n}{k-2}$.
Therefore, $|\widehat{\mathcal{H}}_{3}| \le \frac{1}{c(k,s)}\binom{n}{k-2}$.

Next we show that $\widehat{\mathcal{H}}_{2} = O(n^{k-2})$.
By Lemma \ref{LEMMA-H-covered-Ai} and Equation $(\ref{EQU-weight-sum})$,
\begin{align}
|\partial\mathcal{H}| =
\sum_{E\in \mathcal{H}}\omega_{\mathcal{H}}(E)
 = \sum_{E\in \widehat{\mathcal{H}}_{1}}\omega_{\mathcal{H}}(E)
+ \sum_{E\in \widehat{\mathcal{H}}_{2}}\omega_{\mathcal{H}}(E)
 \ge |\widehat{\mathcal{H}}_{1}| + \frac{k}{k-1}|\widehat{\mathcal{H}}_{2}|. \notag
\end{align}
Therefore, $|\widehat{\mathcal{H}}_{1}| + \frac{k}{k-1}|\widehat{\mathcal{H}}_{2}| \le \binom{n}{k-1}$,
which implies that
\begin{align}
|\widehat{\mathcal{H}}_{2}|
& = (k-1)\left(|\widehat{\mathcal{H}}_{1}| + \frac{k}{k-1}|\widehat{\mathcal{H}}_{2}|
    + |\widehat{\mathcal{H}}_{3}| - |\mathcal{H}|\right) \notag\\
& \le (k-1)\left(\binom{n}{k-1} + |\widehat{\mathcal{H}}_{3}| - (1-\epsilon)\binom{n-1}{k-1} \right)
< 2 \epsilon k \binom{n-1}{k-1}. \notag
\end{align}
This completes the proof of Claim \ref{CLAIM-H2+H3}.
\end{proof}

Note that the proof of Claim \ref{CLAIM-H2+H3} also shows that
\begin{align}\label{INEQU-upper-H}
|\mathcal{H}|
\le |\widehat{\mathcal{H}}_{1}| + \frac{k}{k-1}|\widehat{\mathcal{H}}_{2}| + |\widehat{\mathcal{H}}_{3}|
\le \binom{n}{k-1} + O(n^{k-2}).
\end{align}

Claim \ref{CLAIM-H2+H3} implies that
\begin{align}\label{INEQU-low-bound-hat-H1}
|\widehat{\mathcal{H}}_{1}| =  |\mathcal{H}| - \left(|\widehat{\mathcal{H}}_{2}|+|\widehat{\mathcal{H}}_{3}| \right)
> (1-4\epsilon k)\binom{n-1}{k-1}.
\end{align}

By definition,
for every $E\in \widehat{\mathcal{H}}_{1}$ there exists a unique $s$-homogeneous hypergraph $\mathcal{H}_{i}$ for some $i$
such that $E\in \mathcal{H}_i$,
moreover, $r(\mathcal{J}_i) = k-1$ and $\mathcal{J}_i$ contains exactly $k-1$ $(k-1)$-sets.
Therefore, $\mathcal{I}(E,\mathcal{H}_i)$ contains a unique vertex $c \in E$ such that
every $(k-1)$-subset of $E$ that contains $c$ is contained in $\mathcal{I}(E,\mathcal{H}_i)$.
Let $c(E)$ denote this unique vertex $c$ for every $E\in \widehat{\mathcal{H}}_{1}$.
Define $\mathcal{G}_{i} = \left\{E\in \widehat{\mathcal{H}}_{1}: c(E) = i \right\}$ for $i \in [n]$,
and notice that $\bigcup_{i\in[n]}\mathcal{G}_i = \widehat{\mathcal{H}}_{1}$ is a partition.
Let $\mathcal{G}_{i}(i) = \left\{E\setminus\{i\}: E\in \mathcal{G}_i \right\}$ for $i\in[n]$.
From the proof of Lemma \ref{LEMMA-rank-k-1} $(1)$, for every $i\in [n]$ and every $E\in \mathcal{G}_i$
the set $E\setminus\{i\}$ is not contained in any set in $\mathcal{H}\setminus\{E\}$.
Therefore, $\mathcal{G}_i(i) \cap \mathcal{G}_j(j) = \emptyset$ for all $\{i,j\} \subset [n]$.

\begin{claim}\label{CLAIM-shadow-disjoint}
$\partial\mathcal{G}_i(i) \cap \partial\mathcal{G}_{j}(j) = \emptyset$ for all $\{i,j\} \subset [n]$.
\end{claim}
\begin{proof}[Proof of Claim \ref{CLAIM-shadow-disjoint}]
Suppose not.
Without loss of generality we may assume that there exists
$A \in \partial\mathcal{G}_1(1) \cap \partial\mathcal{G}_{2}(2)$.
Then there exists $E_1 \in \mathcal{G}_1$ and $E_2 \in \mathcal{G}_2$ such that
$E_1 = \{1,u\}\cup A$ and $E_2 = \{2,v\} \cup A$ for some $u,v \in [n]$.
Since $\mathcal{G}_1$ is $s$-homogeneous and $|E_2\cap E_1| \ge k-2 \ge k-a_1$,
by Lemma \ref{LEMMA-large-intersection}, $1 \in E_2$.
Similarly, we obtain $2 \in E_1$.
Therefore, $E_1 = E_2 = \{1,2\} \cup A$, which implies that $\{1,2\} \cup A \in \mathcal{G}_1\cap \mathcal{G}_2$,
a contradiction.
\end{proof}

Let $x_i \in \mathbb{R}$ such that $|\mathcal{G}_{i}| = |\mathcal{G}_{i}(i)| = \binom{x_i}{k-1}$ for $i \in [n]$.
Without loss of generality we may assume that $x_1 \ge \cdots \ge x_n \ge 0$.
By the Kruskal-Katona theorem (e.g. see \cite{LO93}),
\begin{align}
|\mathcal{G}_i(i)|
\le \frac{\binom{x_i}{k-1}}{\binom{x_i}{k-2}} |\partial\mathcal{G}_i(i)|
= \frac{x_i - k +2}{k-1} |\partial\mathcal{G}_i(i)|, \notag
\end{align}
for $i \in [n]$.
Therefore by $(\ref{INEQU-low-bound-hat-H1})$ and Claim \ref{CLAIM-shadow-disjoint},
\begin{align}
(1-4\epsilon k)\binom{n-1}{k-1}
 <  |\widehat{\mathcal{H}}_1|
& = \sum_{i\in [n]}|\mathcal{G}_i|
= \sum_{i\in [n]}|\mathcal{G}_i(i)|
\le  \sum_{i\in \mathcal{H}}\frac{x_i - k +2}{k-1} |\partial\mathcal{G}_i(i)| \notag\\
& \le \frac{x_1 - k +2}{k-1} \sum_{i\in \mathcal{H}} |\partial\mathcal{G}_i(i)|
\le \frac{x_1 - k +2}{k-1} \binom{n}{k-2}, \notag
\end{align}
which implies that
\begin{align}
x_1
\ge (k-1)\frac{(1-4\epsilon k)\binom{n-1}{k-1}}{\binom{n}{k-2}}+k-2
> (1-5\epsilon k)n.\notag
\end{align}
Therefore,
\begin{align}
|\mathcal{G}_1| = \binom{x_1}{k-1} > \binom{(1-5\epsilon k)n}{k-1} > (1-5\epsilon k^2) \binom{n-1}{k-1},\notag
\end{align}
which together with $(\ref{INEQU-upper-H})$
implies that all but at most $5\epsilon k^2 n^{k-1}$ edges in $\mathcal{H}$ contain the vertex $1$.
\end{proof}

Now we prove Theorem \ref{THM-xd-cluster-free}.

\begin{proof}[Proof of Theorem \ref{THM-xd-cluster-free}]
Let $d \ge p \ge 2$, $k > p$, $s=kd+1$,
and $\vec{a} = (a_1,\ldots,a_p)$ be a sequence of positive integers such that $a_1 \ge \cdots \ge a_p$ and $\sum_{i\in[p]}a_i = k$.
Let $n \ge n_0(k,d)$ be sufficiently large.
Let $\mathcal{H} \subset \binom{[n]}{k}$ be a hypergraph that contains no $(\vec{a},d)$-$\Delta$-systems
and $|\mathcal{H}| = \binom{n-1}{k-1}$.
It suffices to show that all edges in $\mathcal{H}$ contain a fixed vertex.

From the proof of Theorem \ref{THM-stability} we know that
$\mathcal{H}$ contains a subgraph $\mathcal{G}_1$ such that
all edges in $\mathcal{G}_1$ contains a fixed vertex (we may assume that this vertex is $1$),
moreover, $\mathcal{G}_1$ consists of
pairwise edge-disjoint $s$-homogeneous hypergraphs whose intersection patterns have rank $k-1$
and contain all $(k-1)$-subsets of $[k]$ that contain $1$.

%

Define
\begin{align}
\mathcal{B}_0 &= \{E\in \mathcal{H}: 1\not\in E\}, \notag\\
\mathcal{B}_1 &= \{E\in \mathcal{H}: 1 \in E \text{ and } |E\cap B|\ge k-a_1 \text{ for some } B\in \mathcal{B}_0\}, \notag\\
\mathcal{G} &= \{E\in \mathcal{H}\setminus \mathcal{B}_1: 1\in E, \text{ $\forall S\subset E$ with $1\in S$ is the center of a $\Delta$-system in
                $\mathcal{H}$ of size $s$}\}, \notag\\
\mathcal{B}_2 &= \{E\in \mathcal{H}: 1\in E\}\setminus (\mathcal{B}_1 \cup \mathcal{G}). \notag
\end{align}
Note that $\mathcal{G}_1\subset \mathcal{G}$.
Let
\begin{align}
\mathcal{B}_1(1) = \{E\setminus 1: E\in \mathcal{B}_1\},
\quad
\mathcal{G}(1) = \{E\setminus 1: E\in \mathcal{G}\},
\quad {\rm and} \quad
\mathcal{B}_2(1) = \{E\setminus 1: E\in \mathcal{B}_2\}. \notag
\end{align}
Let $\mathcal{B}^{\ast}_{1}(1),\mathcal{B}^{\ast}_{2}(1)$
be maximum $s$-homogeneous subgraphs of $\mathcal{B}_1(1),\mathcal{B}_2(1)$, respectively.
Then by Theorem \ref{THM-FU83}, $|\mathcal{B}^{\ast}_{i}(1)| \ge c(k,s)|\mathcal{B}_i(1)|$ for some constant $c(k,s)>0$ and $i=1,2$.
Recall that for every $E \in \partial\mathcal{G}(1)$, $\deg_{\mathcal{G}(1)}(E)$ is the number of edges in $\mathcal{G}(1)$ that contain $E$.
Since $\sum_{E\in \partial\mathcal{G}(1)}\deg_{\mathcal{G}(1)}(E) = (k-1)|\mathcal{G}(1)|$ and $\deg_{\mathcal{G}(1)}(E) \le n-k+1$,
we have
\begin{align}\label{INEQU-G-shadow-G}
|\partial\mathcal{G}(1)| \ge \frac{k-1}{n-k+1}|\mathcal{G}(1)|.
\end{align}

\begin{claim}\label{CLAIM-G-B0}
$|\mathcal{G}| + 4 |\mathcal{B}_0| \le \binom{n-1}{k-1}$.
\end{claim}
\begin{proof}[Proof of Claim \ref{CLAIM-G-B0}]
Notice that by definition $|E\cap B| \le k-a_1-1 \le k-3$ for all $E\in \mathcal{G}(1)$ and $B \in \mathcal{B}_0$.
Therefore, $\partial\mathcal{G}(1) \cap \partial_{2}\mathcal{B}_0 = \emptyset$,
and hence $|\partial\mathcal{G}(1)| + |\partial_{2}\mathcal{B}_0| \le \binom{n-1}{k-2}$.
Let $x \in \mathbb{R}$ such that $|\partial\mathcal{B}_0| = \binom{x}{k-1}$,
then by the Kruskal-Katona theorem and Proposition \ref{PROP-shadow-H-H-F-free},
\begin{align}
|\partial_2\mathcal{B}_0|
\ge \frac{k-1}{x-k+1}|\partial\mathcal{B}_0|
\ge \frac{k-1}{x-k+1}c(k,s)|\mathcal{B}_0|. \notag
\end{align}
Therefore, together with $(\ref{INEQU-G-shadow-G})$ we obtain
\begin{align}
\frac{k-1}{n-k+1}|\mathcal{G}(1)| + \frac{k-1}{x-k+1}c(k,s)|\mathcal{B}_0| \le \binom{n-1}{k-2}, \notag
\end{align}
which implies $|\mathcal{G}| + c(k,s)\frac{n-k+1}{x-k+1}|\mathcal{B}_0| \le \binom{n-1}{k-1}$.
By Theorem \ref{THM-stability}, $\binom{x}{k-1} = |\partial\mathcal{B}_0| \le k |\mathcal{B}_0| \le \delta n^{k-1}$
for all sufficiently small $\delta>0$ (as long as $n$ is sufficiently large),
so $x < \delta' n$ for some sufficiently small $\delta' > 0$ (depending on $\delta$).
Choosing $\delta' \ll c(k,s)$ we obtain
$c(k,s)\frac{n-k+1}{\delta'n -k+1} > 4$, this completes the proof of Claim \ref{CLAIM-G-B0}.
\end{proof}

\begin{claim}\label{CLAIM-B1}
Every $E\in \mathcal{B}^{\ast}_{1}(1)$ has a $(k-2)$-subset that is not contain in any other set
in $\mathcal{B}^{\ast}_{1}(1) \cup \mathcal{G}'$.
\end{claim}
\begin{proof}[Proof of Claim \ref{CLAIM-B1}]
Suppose not.
Let $E = \{v_1,\ldots,v_{k-1}\} \in \mathcal{B}^{\ast}_{1}(1)$ such that
$E\setminus\{v_i\}$ is contained in some set in $\mathcal{B}^{\ast}_{1}(1) \cup \mathcal{G}(1)$ for $1\le i \le k-1$.
Without loss of generality we may assume that
$E\setminus \{v_i\} \in \mathcal{I}(E,\mathcal{G}(1))$ for $1 \le i \le \ell$,
and $E\setminus \{v_i\} \in \mathcal{I}(E,\mathcal{B}^{\ast}_1(1))$ for $\ell+1 \le i \le k-1$.

Let $\mathcal{J}_{\mathcal{B}^{\ast}_{1}(1)}$ be the intersection pattern of $\mathcal{B}^{\ast}_{1}(1)$.
Let $\mathcal{B}^{\ast}_{1} = \{E\cup\{1\}: E\in \mathcal{B}^{\ast}_{1}(1)\}$,
and note that $\mathcal{B}^{\ast}_{1}$ is also $s$-homogeneous with intersection pattern
$\mathcal{J}_{\mathcal{B}^{\ast}_{1}} := \{A\cup \{1\}: A\in \mathcal{J}_{\mathcal{B}^{\ast}_{1}(1)}\}$.
Let $\widehat{E} = E\cup \{1\} \in \mathcal{B}^{\ast}_{1}$.

If $\ell = 0$, then $\mathcal{J}_{\mathcal{B}^{\ast}_{1}(1)} = 2^{[k-1]}\setminus \{[k-1]\}$,
and hence $r(\Pi(\mathcal{I}(\widehat{E},\mathcal{B}^{\ast}_{1}))) = k-1$ and $\Pi(\mathcal{I}(\widehat{E},\mathcal{B}^{\ast}_{1}))$
contains all $(k-1)$-subsets of $\widehat{E}$ that contain $1$.
By definition there exists $B \in \mathcal{B}_0$ such that $|B\cap \widehat{E}|\ge k-a_1$.
However, by Lemma \ref{LEMMA-large-intersection}, $1 \in B$, a contradiction.
Therefore, $\ell \ge 1$.

Let $E_i \in \mathcal{G}$ such that $E_i \cap \widehat{E} = \widehat{E}\setminus\{v_i\}$ for $1\le i \le \ell$.
Let $B \in \mathcal{B}_0$ such that $|B\cap \widehat{E}|\ge k-a_1$ and suppose that $|B\cap \widehat{E}| = k-t$ for some $1\le t \le a_1$.
Then for $1\le i \le \ell$ we have $|B\cap E_i| \ge k-t-1$.
However, by the definition of $\mathcal{G}$, $|B \cap E_i| \le k-a_1-1$ for $1\le i \le \ell$.
Therefore, $|B\cap \widehat{E}| = k-a_1$ and $v_i \in B$ for all $1 \le i \le \ell$.
Let $\bigcup_{i\in [p]}A_i = \widehat{E}$ be an $\vec{a}$-partition such that $A_1 = \widehat{E}\setminus B$.
Note that for $2\le i \le p$,
either $1 \in \widehat{E}\setminus A_i \subset E_{j_i}$ for some $1 \le j_i \le \ell$, which by the definition of $\mathcal{G}$,
is the center of some $\Delta$-system of size $s$ in $\mathcal{H}$,
or $\{1,v_1, \ldots, v_{\ell}\} \subset \widehat{E}\setminus A_i$,
which implies that $\widehat{E}\setminus A_i \in \mathcal{I}(\widehat{E},\mathcal{B}^{\ast}_1)$
and hence is the center of some $\Delta$-system of size $s$ in $\mathcal{B}^{\ast}_1$.
Note that $E\setminus A_1$ is the center of a $\Delta$-system of size $2$, i.e. $\{\widehat{E},B\}$.
Therefore, by Lemma \ref{LEMMA-semi-implies-nonsemi}, $\mathcal{H}$ contains an $(\vec{a},d)$-$\Delta$-system, a contradiction.
\end{proof}

By Claim \ref{CLAIM-B1}, we obtain $|\partial\mathcal{G}(1)| + |\mathcal{B}^{\ast}_{1}(1)| \le \binom{n-1}{k-2}$,
which implies $|\mathcal{G}| + c(k,s)\frac{n-k+1}{k-1}|\mathcal{B}_1|  \le \binom{n-1}{k-1}$.
Note that $c(k,s)\frac{n-k+1}{k-1} \gg 1$, so
\begin{align}
|\mathcal{G}| + 4|\mathcal{B}_1|  \le \binom{n-1}{k-1}. \notag
\end{align}

\begin{claim}\label{CLAIM-B2}
Every $E\in \mathcal{B}^{\ast}_{2}(1)$ has a $(k-2)$-subset that is not contain in any other set
in $\mathcal{B}^{\ast}_{2}(1) \cup \mathcal{G}'$.
\end{claim}
\begin{proof}[Proof of Claim \ref{CLAIM-B2}]
Suppose not.
Let $E = \{v_1,\ldots,v_{k-1}\} \in \mathcal{B}^{\ast}_{2}(1)$ such that
$E\setminus\{v_i\}$ is contained in some set in $\mathcal{B}^{\ast}_{2}(1) \cup \mathcal{G}(1)$ for $1\le i \le k-1$.
Without loss of generality we may assume that
$E\setminus \{v_i\} \in \mathcal{I}(E,\mathcal{G}(1))$ for $1 \le i \le \ell$,
and $E\setminus \{v_i\} \in \mathcal{I}(E,\mathcal{B}^{\ast}_2(1))$ for $\ell+1 \le i \le k-1$.

Let $\mathcal{J}_{\mathcal{B}^{\ast}_{2}(1)}$ be the intersection pattern of $\mathcal{B}^{\ast}_{2}(1)$.
Let $\mathcal{B}^{\ast}_{2} = \{E\cup\{1\}: E\in \mathcal{B}^{\ast}_{2}(1)\}$,
and note that $\mathcal{B}^{\ast}_{2}$ is also $s$-homogeneous with intersection pattern
$\mathcal{J}_{\mathcal{B}^{\ast}_{2}} := \{A\cup \{1\}: A\in \mathcal{J}_{\mathcal{B}^{\ast}_{2}(1)}\}$.
Let $\widehat{E} = E\cup \{1\} \in \mathcal{B}^{\ast}_{2}$.

If $\ell = 0$, then $\mathcal{J}_{\mathcal{B}^{\ast}_{2}(1)} = 2^{[k-1]}\setminus \{[k-1]\}$,
and hence $r(\Pi(\mathcal{I}(\widehat{E},\mathcal{B}^{\ast}_{2}))) = k-1$ and $\Pi(\mathcal{I}(\widehat{E},\mathcal{B}^{\ast}_{2}))$
contains all $(k-1)$-subsets of $\widehat{E}$ that contain $1$.
Since $\mathcal{I}(\widehat{E},\mathcal{B}^{\ast}_{2})$ is closed under intersection,
all proper subsets of $\widehat{E}$ that contain $1$ is contained in $\mathcal{I}(\widehat{E},\mathcal{B}^{\ast}_{2})$,
which by definition, implies that $\widehat{E} \in \mathcal{G}$, a contradiction.
Therefore, $\ell \ge 1$.

Let $E_i \in \mathcal{G}$ such that $E_i \cap \widehat{E} = \widehat{E}\setminus\{v_i\}$ for $1\le i \le \ell$.
For every proper subset $S\subset \widehat{E}$ with $1\in S$,
if $v_i \not\in S$ for some $1\le i \le \ell$, then $S\subset E_i$,
which, by the definition of $\mathcal{G}$,
means that $S$ is the center of some $\Delta$-system of size $s$ in $\mathcal{H}$.
If $\{v_1,\ldots, v_{\ell}\} \subset S$, then $S \in \mathcal{I}(\widehat{E},\mathcal{B}^{\ast}_{2})$
and hence $S$ is the center of some $\Delta$-system of size $s$ in $\mathcal{B}^{\ast}_2$.
Therefore, every proper subset $S\subset \widehat{E}$ with $1\in S$
is the center of some $\Delta$-system of size $s$ in $\mathcal{H}$,
which by definition, implies that $\widehat{E} \in \mathcal{G}$, a contradiction.
\end{proof}

Similarly, we obtain
\begin{align}
|\mathcal{G}| + 4|\mathcal{B}_2|  \le \binom{n-1}{k-1}. \notag
\end{align}

Therefore, by the assumption that $|\mathcal{H}| = \binom{n-1}{k-1}$ we obtain
\begin{align}
3\binom{n-1}{k-1}
& \le 3|\mathcal{H}| + |\mathcal{B}_0| + |\mathcal{B}_1| + |\mathcal{B}_2| \notag\\
& = |\mathcal{G}|+ 4|\mathcal{B}_0| +  |\mathcal{G}|+ 4|\mathcal{B}_1|  +  |\mathcal{G}|+ 4|\mathcal{B}_2|
\le 3\binom{n-1}{k-1}, \notag
\end{align}
which implies that $|\mathcal{G}| = \binom{n-1}{k-1}$ and
$\mathcal{B}_0 = \mathcal{B}_1 = \mathcal{B}_2 = \emptyset$.
This completes the proof of Theorem \ref{THM-xd-cluster-free}.
\end{proof}
\bibliographystyle{abbrv}
\bibliography{non_trivial_intersecting}
\end{document}